\documentclass[11pt, reqno]{amsart}

\usepackage[english]{babel}
\usepackage[left=3.5cm, right=3.5cm, top=3.1cm, bottom=3cm]{geometry}
\usepackage{url, amsfonts, amsmath, amsthm, amssymb, mathtools, mathrsfs, enumerate}
\usepackage{color, xcolor, verbatim, tensor, tikz, tikz-cd}
\usetikzlibrary{angles,quotes}
\usepackage[colorlinks=true,citecolor=blue]{hyperref}
\hypersetup{linkcolor=magenta, linktoc=page}
\usetikzlibrary{matrix,calc}
\definecolor{wine-stain}{rgb}{0.5,0,0}

\newcommand{\red}[1]{\textcolor{red}{#1}}

\newtheorem{thm}{Theorem}[section]
\newtheorem{prop}[thm]{Proposition}
\newtheorem{lem}[thm]{Lemma}
\newtheorem{cor}[thm]{Corollary}

\theoremstyle{definition}

\newtheorem{rem}[thm]{Remark}

\newtheorem*{ques*}{Question}
\newtheorem*{thm*}{Theorem}
\newtheorem*{rem*}{Remark}
\newtheorem*{rems*}{Remarks}
\newtheorem*{exs*}{Examples}
\newtheorem*{mthm*}{Main Theorem}

\numberwithin{equation}{section}

\newcommand{\ep}{\varepsilon}
\newcommand{\vp}{\varphi}

\renewcommand{\d}{\partial}

\newcommand{\RR}{\mathbb{R}}

\newcommand{\sL}{\mathcal{L}}

\newcommand{\fX}{\mathfrak{X}}

\newcommand{\dvol}{\mathrm{dvol}}

\newcommand{\Rc}{\mathrm{Ric}}

\makeatletter
\renewcommand*{\eqref}[1]{%
	\hyperref[{#1}]{\textup{\tagform@{\ref*{#1}}}}%
}
\makeatother

\title[Killing property of the defining vector field for almost Yamabe solitons]{On the Killing property of the defining vector field for an almost Yamabe soliton}
\author[Ramesh Mete]{Ramesh Mete}
\address{Department of Mathematics, Indian Institute of Technology Bombay, Powai, Mumbai - 400076, India.}
\email{ramesh2025m@gmail.com, rameshm@math.iitb.ac.in}
\date{\today}

\begin{document}
	
\maketitle

\begin{abstract}
In this paper, we first investigate almost Yamabe solitons on compact Riemannian manifolds without boundary of dimension greater than or equal to two. We provide some sufficient conditions for which the defining conformal vector field associated to a compact almost Yamabe soliton is a Killing vector field. We then study almost Yamabe solitons on complete, non-compact Riemannian manifolds. We prove the Killing property of the defining conformal vector field associated to a complete, non-compact almost Yamabe soliton under certain conditions when the dimension is strictly greater than two.
\end{abstract}
	

\section{Introduction}

Following \cite{Barbosa-Ribeiro Jr-2013}, a Riemannian manifold $(M, g)$ together with a smooth vector field $X \in \fX(M)$ is called an {\em almost Yamabe soliton} if there exists a smooth real-valued function $\rho: M \to \RR$ satisfying
\begin{equation}\label{eq:almost-Yam-soliton}
\frac{1}{2} \sL_{X} g = (R - \rho) g,
\end{equation}
where $\sL_{X}$ denotes the Lie derivative along the vector field $X$ and $R$ stands for the scalar curvature of the Riemannian metric $g$. We will denote an almost Yamabe soliton by the $4$-tuple $(M, g, X, \rho)$. An almost Yamabe soliton $(M, g, X, \rho)$ is called a {\em gradient almost Yamabe soliton} if there exists a real-valued smooth function $f$ on $M$ such that $X = \nabla f$, where $\nabla$ denotes the Levi-Civita connection of the Riemannian manifold $(M, g)$. In this case, the equation \eqref{eq:almost-Yam-soliton} reduced to the following condition
\begin{equation}\label{eq:grad-almost-Yamabe-soliton}
\text{Hess}(f) = (R - \rho) g,
\end{equation} 
where $\text{Hess}(f)$ denotes the Hessian of $f$. The above equation can be expressed in local coordinates as follows:
\begin{equation}\label{eq:grad-almost-Yamabe-soliton-local-expression}
\nabla_{i} \nabla_{j} f = (R - \rho) g_{i j}.
\end{equation}
If the function $\rho$ is {\em constant} on $M$, then an almost Yamabe soliton is said to be a {\em Yamabe soliton}, and if in addition $X = \nabla f$, then we say that the triple $(M, g, \nabla f)$ is a {\em gradient Yamabe soliton}.

\vspace*{1mm}
Recall that a smooth vector field $X \in \fX(M)$ on a Riemannian manifold $(M, g)$ is called {\em conformal} if it satisfies
\begin{equation}\label{eq:definition-of-conformal-vector-fields}
\frac{1}{2} \sL_{X} g = \psi g,
\end{equation}
for some real-valued smooth function $\psi$ on $M$. A conformal vector field $X$ is called {\em homothetic} if the conformal factor $\psi$ is a constant function, and is called a {\em Killing vector field} if $\psi$ is identically zero on $M$. A non-Killing vector field will be called a {\em non-trivial} conformal vector field. Observe that if $(M, g, X, \rho)$ is an almost Yamabe soliton, then $X$ is a conformal vector field with conformal factor $(R - \rho)$. We say that vector field $X \in \fX(M)$ is {\em closed} if it satisfies the following condition:
\begin{equation}\label{eq:definition-of-closed-vector-fields}
\nabla_{Y} X = \psi Y
\end{equation}
for all smooth vector fields $Y \in \fX(M)$ and some function $\psi \in C^\infty(M, \RR)$. A closed vector field is said to be {\em parallel} if the function $\psi$ is identically zero. Since for any smooth vector fields $X, Y, Z \in \fX(M)$ the following identity holds:
\begin{equation}\label{eq:Lie-derivative-formula-in-terms-of-covariant-derivative}
\sL_{X} g (Y, Z ) = g(\nabla_Y X, Z) + g(Y, \nabla_Z X),
\end{equation}
one can see that any closed vector field is conformal. In particular, any parallel vector field is Killing. On the other hand, any closed Killing vector field must be parallel.

\vspace*{1mm}
We now recall some well-known results about almost Yamabe solitons. It was proved in \cite[Proposition 2.1]{Barbosa-Ribeiro Jr-2013} that almost Yamabe solitons are conformal solutions of the non-normalized Yamabe flow, introduced by Hamilton \cite{Hamilton-1988}, which is defined by
\begin{equation}\label{eq:Yamabe-flow-non-normalized}
\frac{\d }{\d t} g(t) = - R_{g(t)} g(t),
\end{equation}
where $g(t)$ is a path of Riemannian metrics on $M$ and $R_{g(t)}$ denotes the scalar curvature of $g(t)$. The solution $g(t)$ is said to be {\em conformal} if $g(t) = \sigma(x, t) \vp_{t}^{\ast} g(0)$ for some positive function $\sigma : M \times [0, \ep) \to \RR$ and diffeomorphism $\vp_t : M \to M$ for all $t \in [0,\ep)$ with $\vp_0 = \text{Id}_{M}$ and $\sigma(x, 0) = 1$ for all $x\in M$, where $\ep > 0$. Using the Yamabe flow \eqref{eq:Yamabe-flow-non-normalized}, di Cerbo and Disconzi \cite{di Cerbo-Disconzi-2008} proved that a compact gradient Yamabe soliton has constant scalar curvature. See also \cite{Chow-Lu-Ni-book-chap} and \cite{Hsu-2012} for different proofs of this well-known fact. 
	
\vspace*{1mm}
The following result is due to Barbosa and Ribeiro Jr \cite{Barbosa-Ribeiro Jr-2013} which give some sufficient criterions for Killing-ness of the defining smooth vector field $X$.
	
\begin{thm}[{\cite[Theorem 1.3]{Barbosa-Ribeiro Jr-2013}}]
\label{thm:Barbosa-Ribeiro Jr}
Let $(M^n , g, X, \rho)$ be an almost Yamabe soliton. Then $X$ is a Killing vector field if one of the following conditions holds:
\begin{enumerate}[~~~1.]
\item $M$ is compact and 
\begin{align}\label{integral-condition-involving-Ric-BarRib2013}
\int_M \big(\Rc(X, X) - (n-2)\langle \nabla\rho, X \rangle\big) \dvol_{g} \leq 0.
\end{align}
			
\item $M$ is compact and 
\begin{align}\label{integral-condition-involving-grad of rho-BarRib2013}
\int_M \langle \nabla\rho, X \rangle \dvol_{g} \leq 0.
\end{align}
			
\item $M$ is complete and non-compact, and $|X| \in L^1(M)$ and either $R \geq \rho$ or $R \leq \rho$. 
\end{enumerate}
Here and throughout of the paper, $\langle \cdot, \cdot \rangle := g (\cdot, \cdot)$ and $\dvol_{g}$ denotes the volume element for the Riemannian metric $g$.
\end{thm}
	
\begin{rem}\label{rem:errors-in-Barbosa-Ribeiro Jr}
Note that there are mistakes in the first item and second item of Theorem 1.3 in Barbosa-Ribeiro Jr \cite{Barbosa-Ribeiro Jr-2013}. More precisely, the integral condition in the first item must be the above inequality \eqref{integral-condition-involving-Ric-BarRib2013} instead of the following incorrect condition $$\int_M \Big(\frac{1}{2}\Rc(X, X) + (n-2)\langle \nabla\rho, X \rangle\Big) \dvol_{g} \leq 0$$ given in \cite[Theorem 1.3 (1.)]{Barbosa-Ribeiro Jr-2013}. On the other hand, the integral condition in the second item should be the inequality \eqref{integral-condition-involving-grad of rho-BarRib2013} instead of the following incorrect condition
$$\int_M \langle \nabla\rho, X \rangle \dvol_{g} \geq 0$$
mentioned in \cite[Theorem 1.3 (2.)]{Barbosa-Ribeiro Jr-2013}. The reader can refer to Section \ref{sec:some corrections in Barbosa-Ribeiro Jr article} for more details.
\end{rem}

Recently, for dimension $n \geq 3$ Hwang and Yun \cite{Hwang-Yun-2025} gave some sufficient conditions for which the defining vector field $X$ associated to an almost Yamabe soliton is Killing. When the manifold $M$ is compact, they proved the following result.

\begin{thm}[{\cite[Theorem 1.1]{Hwang-Yun-2025}}]
\label{thm:Hwang-Yun-2025-compact Riem mfds}
Let $(M, g, X, \rho)$ be a compact almost Yamabe soliton with $n \geq 3$.
		
\begin{enumerate}[~~~(i)]
\item If either $0 \leq \rho \leq R$, or $\langle X, \nabla\rho \rangle \leq 0$, then $X$ is a Killing vector field.
			
\vspace*{1mm}
\item If $\Rc(X, X) \leq 0$, then $X$ is a parallel Killing vector field.
			
\vspace*{1mm}
\item If $\int_M R^2 = 2 \int_M \rho^2$, then $X$ is a Killing vector field, and $(M, g, X)$ is a Yamabe soliton with $R = \rho = 0$.
\end{enumerate}
\end{thm}
	
We remark that the second item of above Theorem \ref{thm:Hwang-Yun-2025-compact Riem mfds} is basically Corollary 1.4 in \cite{Barbosa-Ribeiro Jr-2013}, which states that {\em a compact Riemannian manifold of negative Ricci curvature can not be a non-trivial almost Yamabe soliton}. The key argument for this fact is that, since the manifold $M$ is compact (and without boundary), one can easily obtain that $\nabla X = 0$ under the Ricci non-positive condition, and hence $X$ is a parallel Killing vector field (see Section \ref{sec:some corrections in Barbosa-Ribeiro Jr article} for more details).

\vspace*{1mm}
Our first main result in this paper is the following, which extends a result of \cite{Hwang-Yun-2025} in the compact case and also holds for dimension $n = 2$.
	
\begin{thm}
\label{thm:main-compact case}
Suppose $(M, g, X, \rho)$ is a compact almost Yamabe soliton of dimension $n \geq 2$ without boundary. If the scalar curvature $R$ of the Riemannian metric $g$ satisfies the following integral inequality
\begin{equation}\label{eq:inequality-assumption-compact-case}
\int_M R^2 \dvol_{g} \geq \int_M \rho^2 \dvol_{g},
\end{equation}
then $X$ is a Killing vector field. Moreover, when the following equality holds
\begin{equation}\label{eq:equality-assumtion-Hwang-Yun-compact-case}
\int_M R^2 \dvol_{g} = 2 \int_M \rho^2 \dvol_{g}
\end{equation}
the triple $(M, g, X)$ is a (compact) Yamabe soliton with $ R = \rho = 0$.
\end{thm}
	
Note that the last conclusion assuming the equality condition \eqref{eq:equality-assumtion-Hwang-Yun-compact-case} in Theorem \ref{thm:main-compact case} is due to Hwang and Yun \cite[Theorem 3.5]{Hwang-Yun-2025} for dimension $n \geq 3$ (see Theorem \ref{thm:Hwang-Yun-2025-compact Riem mfds} above), and they also proved the following upper bound
\begin{align*}
\int_M R^2 \dvol_{g} \leq 2 \int_M \rho^2 \dvol_{g}
\end{align*}
for any compact almost Yamabe soliton. One can easily see that this upper bound also holds for $n=2$. We conjecture that any compact almost Yamabe soliton $(M^n, g, X, \rho)$ satisfying the following condition
\begin{equation}\label{eq:conjecture-under-strict-inequality}
\int_M \rho^2 \dvol_{g} \leq \int_M R^2 \dvol_{g} \leq 2 \int_M \rho^2 \dvol_{g}
\end{equation}
must be a (compact) Yamabe soliton and hence the scalar curvature $R$ of the Riemannian metric $g$ is constant. Furthermore, it will be interesting to study compact almost Yamabe solitons under the following assumption
\begin{equation}\label{eq:inequality-to-study-further}
0 \leq \int_M R^2 \dvol_{g} < \int_M \rho^2 \dvol_{g}.
\end{equation}
	
\vspace*{1mm}
As a corollary of Theorem \ref{thm:main-compact case}, we have the following result.
	
\begin{cor}
\label{cor:corollary-1-compact-case}
Suppose $(M, g, X, \rho)$ is a compact almost Yamabe soliton of dimension $n \geq 2$. If the scalar curvature $R$ of the Riemannian metric $g$ satisfies either $0\leq \rho \leq R$, or $R \leq \rho \leq 0$, then $X$ is a Killing vector field.
\end{cor}
	
We remark that Corollary \ref{cor:corollary-1-compact-case} under the first condition $0 \leq \rho \leq R$ is due to Hwang and Yun \cite[Theorem 3.1]{Hwang-Yun-2025} for dimension $n \geq 3$ (see also Theorem \ref{thm:Hwang-Yun-2025-compact Riem mfds} above). They also proved the following results about the Killing property of the defining vector field $X$ similar to the third item of Theorem \ref{thm:Barbosa-Ribeiro Jr} on complete non-compact Riemannian manifolds of dimension $n \geq 3$.
	
\begin{thm}[{\cite[Theorem 1.2]{Hwang-Yun-2025}}]
\label{thm:Hwang-Yun-2025-non-compact-case-Ric-non-positive}
Let $(M^n, g, X, \rho)$ be a complete non-compact almost Yamabe soliton with $n \geq 3$. If $\Rc(X, X) \leq 0$ and $|X| \in L^2(M)$, then $X$ is a parallel Killing vector field.
\end{thm}
	
\begin{thm}[{\cite[Theorem 1.3]{Hwang-Yun-2025}}]
\label{thm:Hwang-Yun-non-compact-case-rho-and-R-minus-rho-is-non-negative}
Let $(M^n, g, X, \rho)$ be a complete non-compact almost Yamabe soliton with $n \geq 3$. If $0 \leq \rho \leq R$, $\langle X, \nabla R \rangle \geq 0$ and $R + |X| \in L^2(M)$, then $X$ is a Killing vector field. 
\end{thm}
	
We would like to mention that both of these results also hold true for the case $n =2$. Motivated by Theorem \ref{thm:Hwang-Yun-2025-non-compact-case-Ric-non-positive}, it would be interesting to study whether the vector field $X$ associated to a complete non-compact almost Yamabe soliton $(M^n, g, X, \rho)$ is parallel (and hence Killing) or not under the following weaker assumptions $|X|\in L^1(M)$ but $|X|\notin L^2(M)$, and $\Rc(X, X) \leq 0$.

\vspace*{1mm}
We now state our second main result, which is in the non-compact setting.
	
\begin{thm}
\label{thm:main-1-complete-non-cpt-case}
Let $(M^n, g, X, \rho)$ be a complete, non-compact almost Yamabe soliton of dimension $n \geq 3$. Assume the following conditions hold: 
\begin{equation}\label{eq:inner-prod-of-grad-rho-and-X-non-positive}
\langle \nabla\rho, X \rangle \leq 0, ~~ \rho \in L^2(M), ~~ |X| \in L^2(M),
\end{equation}
and
\begin{equation}\label{eq:grad-R-minus-rho-in-L1-space}
R \in L^2(M), ~~ |\nabla(R - \rho)| \in L^1(M).
\end{equation}
Then $X$ is a Killing vector field.
\end{thm}
	
As a corollary of Theorem \ref{thm:main-1-complete-non-cpt-case}, we have the following two results.
	
\begin{cor}
\label{cor:complete-non-compact-case-under-finite-vol-and-Sobolev-condition}
Let $(M^n, g, X, \rho)$ be a complete, non-compact almost Yamabe soliton of dimension $n \geq 3$ with finite volume. Then, $X$ is a Killing vector field under the conditions \eqref{eq:inner-prod-of-grad-rho-and-X-non-positive} and 
\begin{equation}\label{eq:div-of-X-is-in-Sobolev-space}
\mathrm{div}(X) \in W^{1,2}(M).
\end{equation}
Here, $W^{1,2}(M)$ denotes the Sobolev (or Hilbert) space, which is the completion of the vector space of smooth real-valued functions $f$ such that $f \in L^2(M)$ and $|\nabla f|_{g} \in L^2(M)$ with respect to the norm $$\lVert f \rVert_{W^{1,2}(M)} := \left(\lVert f \rVert_{L^2(M)}^{2} + \lVert \nabla f \rVert_{L^2(M)}^{2}\right)^{1/2}.$$
\end{cor}
	
\begin{cor}\label{cor:complete-non-compact-case-under-rho-and-R-minus-rho-non-negative}
Let $(M^n, g, X, \rho)$ be a complete, non-compact almost Yamabe soliton of dimension $n \geq 3$. If $0 \leq \rho \leq R$, $\langle \nabla\rho, X \rangle \leq 0$, $|\nabla(R - \rho)| \in L^1(M)$ and $R + |X| \in L^2(M)$, then $X$ is a Killing vector field.
\end{cor}
	
We remark that one can relate Corollary \ref{cor:complete-non-compact-case-under-rho-and-R-minus-rho-non-negative} with the above Theorem \ref{thm:Hwang-Yun-non-compact-case-rho-and-R-minus-rho-is-non-negative}.

\vspace*{1mm}
\textbf{Brief outline of the paper:} In Section \ref{sec:preliminaries}, we recall some basic properties of conformal vector fields and almost Yamabe solitons. Next, in Section \ref{sec:proof of main result} we prove our main results Theorem \ref{thm:main-compact case} and Corollary \ref{cor:corollary-1-compact-case} in the compact setting. In Section \ref{sec:some corrections in Barbosa-Ribeiro Jr article}, we describe some errors occurred in the article of Barbosa and Ribeiro Jr \cite{Barbosa-Ribeiro Jr-2013} and we correct them accordingly. Lastly, in Section \ref{sec:proof-of-main-results-in-non-compact-case} we prove our results Theorem \ref{thm:main-1-complete-non-cpt-case}, Corollary \ref{cor:complete-non-compact-case-under-finite-vol-and-Sobolev-condition} and Corollary \ref{cor:complete-non-compact-case-under-rho-and-R-minus-rho-non-negative} in the non-compact case.

\vspace*{1mm}
\section{Preliminaries}
\label{sec:preliminaries}
	
In this section, we recall some standard results about almost Yamabe solitons. Before that we have the following formula about the divergence of the Lie derivative $\sL_{X} g$ on a Riemannian manifold $(M, g)$ together with a smooth vector field $X \in \fX(M)$.
	
\begin{lem}[{\cite[Lemma 2.1, Corollary 2.2 and Corollary 2.3]{Petersen-Wylie-2009}}]
\label{lem:divergence-of-Lie-derivative-Petersen-Wylie}
Suppose $X$ is a smooth vector field on a Riemannian manifold $(M, g)$. Then
\begin{equation*}
\mathrm{div}(\sL_{X} g) (X) = \frac{1}{2} \Delta |X|^2 - |\nabla X|^2 + \Rc(X, X) + \nabla_{X}\mathrm{div}(X).
\end{equation*}
In particular, if $X$ is a Killing vector field, then 
\begin{equation*}
\frac{1}{2} \Delta |X|^2 = |\nabla X|^2 - \Rc(X, X).
\end{equation*}
When $X = \nabla f$ is a gradient vector field we have
\begin{equation*}
\mathrm{div}(\sL_{X} g) (Z) = 2 \Rc(Z, X) + 2 \nabla_{Z}\mathrm{div}(X)
\end{equation*}
for any $Z \in \fX(M)$. In particular, if $X = \nabla f$ is a gradient vector field, then we get the following Bochner formula
\begin{equation*}
\frac{1}{2} \Delta |\nabla f|^2 = |\nabla^2 f|^2 + \Rc(\nabla f, \nabla f) + \langle \nabla\Delta f, \nabla f \rangle.
\end{equation*}
\end{lem}
	
Applying the Lemma \ref{lem:divergence-of-Lie-derivative-Petersen-Wylie}, we have the following result due to Barbosa-Ribeiro Jr \cite{Barbosa-Ribeiro Jr-2013} for any almost Yamabe soliton.
	
\begin{lem}[{\cite[Lemma 2.2]{Barbosa-Ribeiro Jr-2013}}]
Let $(M^n, g, X, \rho)$ be an almost Yamabe soliton (i.e. equation \eqref{eq:almost-Yam-soliton} holds). Then 
\begin{equation}\label{eq:Bochner-type-formula-for-almost-Yamabe-solitons}
\frac{1}{2} \Delta |X|^2 + (n -2) \langle \nabla (R - \rho), X \rangle = |\nabla X|^2 - \Rc(X, X).
\end{equation}
\end{lem}
	
We would like to mention that there is an error in \cite[Lemma 2.2]{Barbosa-Ribeiro Jr-2013} namely ``$- \frac{1}{2} \Rc(X, X)$" appears instead of the correct term ``$- \Rc(X, X)$" on the right hand side of the above formula \eqref{eq:Bochner-type-formula-for-almost-Yamabe-solitons}. More precisely, the error occurs in the equation $(2.4)$ of their article \cite{Barbosa-Ribeiro Jr-2013}. One can also see the fifth item of Theorem $2.1$ in \cite{Hwang-Yun-2025}.
	
\vspace*{1mm}
Next, we recall the following basic property for any conformal vector field.
	
\begin{lem}[{\cite{Yano-1970-bk}}]
\label{lem:identity-for-conformal-vector-field-by-Yano}
For any conformal vector field $X$ satisfying \eqref{eq:definition-of-conformal-vector-fields} on a Riemannian manifold $(M,g)$ of dimension $n$, the following holds
\begin{align*}
(n-1)\Delta\psi + \psi R + \frac{1}{2}\langle \nabla R, X \rangle = 0.
\end{align*}
In particular, if $X$ is homothetic, then $\langle \nabla R, X \rangle = - 2 \psi R$.
\end{lem}

\vspace*{1mm}
As in \cite{Hwang-Yun-2025}, for a smooth vector field $X \in \fX(M)$ one can define a {\em skew-symmetric} $(1,1)$-tensor $\Phi_{X}: \fX(M) \to \fX(M)$ by $$\langle \Phi_{X} (Y), Z \rangle = \frac{1}{2} d X^{\flat} (Y, Z)$$ for $Y, Z \in \fX(M)$, where the standard notation $X^{\flat}$ denotes the dual $1$-form corresponding to the smooth vector field $X$ with respect to the Riemannian metric $g$ in $M$. Then, using \eqref{eq:Lie-derivative-formula-in-terms-of-covariant-derivative} we have the following basic identity
\begin{align*}
\sL_{X} g (Y, Z) + 2 \langle \Phi_{X} (Y), Z \rangle = \sL_{X} g (Y, Z) + d X^{\flat} (Y, Z) = 2 \langle \nabla_{Y} X, Z \rangle 
\end{align*}
for any vector fields $X, Y, Z \in \fX(M)$. In particular, for an almost Yamabe soliton $(M, g, X, \rho)$ one has
\begin{equation}
\nabla_{Y} X = (R - \rho) Y + \Phi_{X}(Y)
\end{equation}
for all smooth vector fields $Y$ on $M$. As a consequence, we see that the smooth vector field $X$ is closed if and only if the $(1,1)$-tensor $\Phi_{X}$ is identically zero. The reader can refer to Theorem $2.1$ of \cite{Hwang-Yun-2025} for more properties about almost Yamabe solitons.

\vspace*{1mm}
We end this section by recalling the following result due to Caminha \cite{Caminha 2011}. Note that the proof of the third item of Theorem \ref{thm:Barbosa-Ribeiro Jr} is based on this result (see Section \ref{sec:some corrections in Barbosa-Ribeiro Jr article}). 
	
\begin{prop}[{\cite[Proposition 2.1]{Caminha 2011}}]
\label{prop:Caminha-2011-triviality-of-div}
Let $X$ be a smooth vector field on the complete, non-compact, oriented Riemannian manifold $M^n$, such that $\mathrm{div}(X)$ does not change sign on $M$. If $|X| \in L^1(M)$, then $\mathrm{div}(X) = 0$ on $M$.
\end{prop}

\vspace*{1mm}
\section{Proof of the main results: compact case}
\label{sec:proof of main result}
	
We now prove our first main Theorem \ref{thm:main-compact case}.
	
\begin{proof}[\textbf{Proof of Theorem \ref{thm:main-compact case}}]
Since $(M, g, X, \rho)$ is an almost Yamabe soliton (i.e. equation \eqref{eq:almost-Yam-soliton} holds) of dimension $n \geq 2$, we see that $X$ is a conformal vector field with conformal factor $(R - \rho)$. Since the manifold $M$ is compact, applying Theorem II.9 of Bourguignon and Ezin \cite[pp. 727]{Bourguignon-Ezin-1987} the vector field $X$ satisfies the following condition
\begin{equation}\label{eq:Bourguignon-Ezin-identity-about-conformal-vector-field}
\int_M \langle \nabla R, X \rangle \dvol_{g} = \int_M X(R) \dvol_{g} = 0.
\end{equation}
Moreover, taking trace in the equation \eqref{eq:almost-Yam-soliton} we also have 
\begin{align}\label{eq:divergence-of-X-for-almost-Yamabe-soliton}
\mathrm{div} (X) = n (R - \rho).
\end{align}
Combing \eqref{eq:divergence-of-X-for-almost-Yamabe-soliton} and \eqref{eq:Bourguignon-Ezin-identity-about-conformal-vector-field} it yields that
\begin{equation}\label{eq:integration-of-R(R-rho)-is-zero}
\begin{split}
\int_M R (R - \rho) \dvol_{g} &= \frac{1}{n} \int_M R  \mathrm{~ div}(X) \dvol_{g} \\
& = - \frac{1}{n} \int_M \langle \nabla R, X \rangle \dvol_{g} \\
& = 0.
\end{split}
\end{equation}
In the second line of above computation we used the formula
\begin{align}\label{eq:div-formula-for-R-times-vector-field}
\mathrm{div}(R X) = \langle \nabla R , X \rangle + R \mathrm{div}(X). 
\end{align}
Therefore, from \eqref{eq:integration-of-R(R-rho)-is-zero} we get that
\begin{equation}\label{eq:integration-of-(R-rho)-square}
\begin{split}
\int_M (R - \rho)^2 \dvol_{g} &= \int_M R(R - \rho) \dvol_{g} - \int_M \rho(R - \rho) \dvol_{g} \\ 
& = - \int_M \rho(R - \rho) \dvol_{g}.
\end{split}
\end{equation}
Now, by the hypothesis \eqref{eq:inequality-assumption-compact-case} we obtain
\begin{align*}
0 &\leq \int_M (R^2 - \rho^2) \dvol_g \\
& = \int_M (R + \rho) (R - \rho) \dvol_{g} \\
& = \int_M R(R - \rho) \dvol_{g} + \int_M \rho(R - \rho) \dvol_{g} \\
&= \int_M \rho(R - \rho) \dvol_{g}, \hspace*{4cm} (\text{by} ~ \eqref{eq:integration-of-R(R-rho)-is-zero}) \\
&= - \int_M (R - \rho)^2 \dvol_{g}  \hspace*{4cm} (\text{by} ~ \eqref{eq:integration-of-(R-rho)-square} 
\end{align*}
and hence we must have $R = \rho$. Therefore, from equation \eqref{eq:almost-Yam-soliton} we conclude that $X$ is a Killing vector field. 
		
Since the condition \eqref{eq:integration-of-R(R-rho)-is-zero} also holds true for dimension $n = 2 $, the conclusion of the second part follows from Theorem 3.5 in \cite{Hwang-Yun-2025}.
\end{proof}

Next, we prove our Corollary \ref{cor:corollary-1-compact-case}.
	
\begin{proof}[\textbf{Proof of Corollary \ref{cor:corollary-1-compact-case}}]
If one of the inequalities $0\leq \rho \leq R$ and $R \leq \rho \leq 0$ holds, then we have 
\begin{align*}
\int_M R^2 \dvol_{g} \geq \int_M \rho^2 \dvol_{g}.
\end{align*}
Therefore, applying Theorem \ref{thm:main-compact case} we conclude that $X$ is a Killing vector field. 
\end{proof}

\vspace*{1mm}
\section{Some corrections in the article of Barbosa-Ribeiro Jr}
\label{sec:some corrections in Barbosa-Ribeiro Jr article}
	
In Remark \ref{rem:errors-in-Barbosa-Ribeiro Jr}, we pointed out the errors in the first and second items of Theorem 1.3 in \cite{Barbosa-Ribeiro Jr-2013}. We now provide a brief description of the proof of Theorem \ref{thm:Barbosa-Ribeiro Jr} correcting the errors appeared in the arguments of the proof of Theorem 1.3 in \cite[pp. 84]{Barbosa-Ribeiro Jr-2013}. We also give an alternative simple proof for the second item of Theorem 1.3 in \cite{Barbosa-Ribeiro Jr-2013}.
	
\begin{proof}[\textbf{Proof of Theorem \ref{thm:Barbosa-Ribeiro Jr}}]
Suppose the manifold $M$ is compact. Then, since $(M, g, X, \rho)$ is an almost Yamabe soliton, integrating the formula \eqref{eq:Bochner-type-formula-for-almost-Yamabe-solitons} we obtain
\begin{align*}
\int_M |\nabla X|^2 \dvol_{g} &= \int_{M} \Rc(X, X) \dvol_{g} + (n-2) \int_M \langle \nabla R, X \rangle \dvol_{g} \\
& \hspace*{3.7cm} - (n-2) \int_M \langle \nabla \rho, X \rangle \dvol_{g} \\
&= \int_M \Big(\Rc(X, X) - (n-2) \langle \nabla \rho, X \rangle \Big)\dvol_{g}. \hspace*{2cm}~ (\text{by}~ \eqref{eq:Bourguignon-Ezin-identity-about-conformal-vector-field})
\end{align*}
Note that a sign error (namely, $(n-2)$ instead of $-(n-2)$) occurs in equation $(3.2)$ of their paper \cite{Barbosa-Ribeiro Jr-2013}. Therefore, under the assumption \eqref{integral-condition-involving-Ric-BarRib2013} we must have $\nabla X = 0$ and hence $X$ is a Killing vector field which is also parallel.
		
For the proof of the second item, using the conditions \eqref{eq:integration-of-(R-rho)-square} and \eqref{eq:divergence-of-X-for-almost-Yamabe-soliton} we compute
\begin{align*}
\int_M (R -\rho)^2 \dvol_{g} &= - \int_M \rho (R -\rho) \dvol_{g} \\
&= - \frac{1}{n} \int_M \rho \mathrm{~ div}(X) \dvol_{g} \\
& = \frac{1}{n} \int_M \langle \nabla \rho, X \rangle \dvol_{g}.
\end{align*}
In the last line, we used the formula 
\begin{equation}\label{eq:divergence-formula-for-multiply-by-functions}
\mathrm{div}(\rho X) = \langle \nabla \rho, X \rangle + \rho \mathrm{div}(X).
\end{equation}
We remark that again a sign error (namely, $-\frac{1}{n}$ instead of $\frac{1}{n}$ in the last line) occurs in their computation after equation $(3.6)$ in \cite[pp. 85]{Barbosa-Ribeiro Jr-2013}. Therefore, under the assumption \eqref{integral-condition-involving-grad of rho-BarRib2013} we must have $R = \rho$ and hence from the equation \eqref{eq:almost-Yam-soliton} we infer that $X$ is a Killing vector field. Note that in this case, $X$ may not be parallel.
		
The proof of the third item was given in \cite[pp. 85]{Barbosa-Ribeiro Jr-2013} by applying Proposition 2.1 of \cite{Caminha 2011} (see Proposition \ref{prop:Caminha-2011-triviality-of-div}). This completes the proof of Theorem \ref{thm:Barbosa-Ribeiro Jr}.
\end{proof}

\vspace*{1mm}
\begin{proof}[\textbf{Alternative proof of Theorem \ref{thm:Barbosa-Ribeiro Jr}(2.)}]
We now provide another simple and direct proof for the second item of Theorem \ref{thm:Barbosa-Ribeiro Jr} when dimension $n \geq 3$. Recall that for any almost Yamabe soliton $(M, g, X, \rho)$, the vector field $X$ satisfies \eqref{eq:definition-of-conformal-vector-fields} with conformal factor $\psi := (R - \rho)$. So, applying Lemma \ref{lem:identity-for-conformal-vector-field-by-Yano} and using the formula \eqref{eq:div-formula-for-R-times-vector-field} we obtain
\begin{align*}
R(R - \rho) &= - (n-1) \Delta (R - \rho) - \frac{1}{2} \langle \nabla R, X \rangle \\
& = - (n-1) \mathrm{div}(\nabla (R- \rho)) - \frac{1}{2} \mathrm{div}(R X) + \frac{R}{2} \mathrm{div}(X) \\
&= - \mathrm{div}\left((n-1)\nabla (R - \rho) + \frac{R}{2} X\right) + \frac{R}{2} \mathrm{div}(X).
\end{align*}
Therefore, using the relation \eqref{eq:divergence-of-X-for-almost-Yamabe-soliton} (i.e. $\mathrm{div}(X) = n (R - \rho)$) and simplifying we get
\begin{equation}\label{equation-for-R-times-R-minus-rho}
R (R -\rho) = \mathrm{div}\left(\frac{R}{n-2} X + \frac{2(n-1)}{n-2} \nabla (R - \rho)\right),
\end{equation}
where we assume that $n \geq 3$. Note that for $n=2$ we have $\Delta (R - \rho) = - \frac{1}{2} \mathrm{div}(R X)$. On the other hand, once again using \eqref{eq:divergence-of-X-for-almost-Yamabe-soliton} and the formula \eqref{eq:divergence-formula-for-multiply-by-functions} we obtain
\begin{equation}\label{equation-for-rho-times-R-minus-rho}
\rho (R - \rho) = \frac{1}{n} \mathrm{div}(\rho X) - \frac{1}{n} \langle \nabla \rho, X \rangle.
\end{equation}
So, subtracting the equation \eqref{equation-for-rho-times-R-minus-rho} from the equation \eqref{equation-for-R-times-R-minus-rho} we have the following formula
\begin{equation}\label{eq:formula-for-R-minus-rho-square-in-terms-of-div}
(R - \rho)^2 = \mathrm{div}\left((\frac{R}{n-2} - \frac{\rho}{n}) X + \frac{2(n-1)}{n-2} \nabla (R - \rho)\right) + \frac{1}{n} \langle \nabla \rho, X \rangle
\end{equation}
for $n \geq 3$. Since the manifold $M$ is closed (i.e. compact and without boundary), integrating the above formula and assuming the condition \eqref{integral-condition-involving-grad of rho-BarRib2013} we must have $R = \rho$, and hence $X$ is a Killing vector field. 
\end{proof}
	
\vspace*{1mm}
We finish this section by correcting the proof of Corollary $1.4$ in \cite[pp. 85]{Barbosa-Ribeiro Jr-2013}. As we see in the above proof, instead of the first and second displayed integral identity in their proof one actually has
\begin{align*}
\int_M |\nabla X |^2 \dvol_{g} = \int_M \Big(\Rc(X, X) - (n-2) \langle \nabla \rho, X \rangle \Big)\dvol_{g}, 
\end{align*}
and 
\begin{align*}
\int_M (R -\rho)^2 \dvol_{g} = \frac{1}{n} \int_M \langle \nabla \rho, X \rangle \dvol_{g}.
\end{align*}
It implies that
\begin{align*}
\int_M \Big(\Rc(X, X) - n(n-2) (R - \rho)^2 - |\nabla X|^2 \Big)\dvol_{g} = 0.
\end{align*}
Note that there is a factor $1/2$ in the term involving the Ricci curvature in their proof which should be removed. Therefore, if we assume that the Ricci curvature is non-positive, we obtain $|\nabla X|^2 = 0$ and hence $X$ is a parallel Killing vector field. This completes the proof of Corollary $1.4$ in \cite{Barbosa-Ribeiro Jr-2013}. See also the proof of Theorem $3.3$ in \cite{Hwang-Yun-2025}.

\vspace*{1mm}
\section{Proof of the main results: non-compact case}
\label{sec:proof-of-main-results-in-non-compact-case}
	
We now prove our second main Theorem \ref{thm:main-1-complete-non-cpt-case}.
	
\begin{proof}[\textbf{Proof of Theorem \ref{thm:main-1-complete-non-cpt-case}}]
For any almost Yamabe soliton $(M, g, X, \rho)$ of dimension $n \geq 3$, we already showed that the identity \eqref{eq:formula-for-R-minus-rho-square-in-terms-of-div} holds. For simplicity, we define a smooth vector field $V \in \fX(M)$ by 
\begin{align*}
V := \left(\frac{R}{n-2} - \frac{\rho}{n}\right) X + \frac{2(n-1)}{n-2} \nabla (R - \rho)
\end{align*}
so that we have
\begin{align}\label{eq:div-for-vector-field-V}
\mathrm{div}(V) = (R - \rho)^2 - \frac{1}{n} \langle \nabla \rho, X \rangle.
\end{align}
Since $\langle \nabla \rho, X \rangle \leq 0$, the divergent of the smooth vector field $V$ is non-negative. On the other hand, if we denote $N_{Y}:= |Y|_{g}$ for any smooth vector field $Y \in \fX(M)$, then using the H\"{o}lder inequality and the Minkowski inequality we obtain that
\begin{align*}
\lVert N_{V} \rVert_{L^1(M)} &\leq \lVert \frac{R}{n-2} - \frac{\rho}{n} \rVert_{L^2(M)} \lVert N_{X} \rVert_{L^2(M)} + \frac{2(n-1)}{n-2} \lVert \nabla (R - \rho) \rVert_{L^1(M)} \\
& \leq \left(\frac{1}{n-2} \lVert R- \rho \rVert_{L^2(M)} + \frac{2}{n(n-2)} \lVert \rho \rVert_{L^2(M)} \right) \lVert N_{X} \rVert_{L^2(M)} \\ & \hspace*{4cm} + \frac{2(n-1)}{n-2} \lVert \nabla (R - \rho) \rVert_{L^1(M)}.
\end{align*}
Here, $\lVert \nabla (R - \rho) \rVert_{L^1(M)}$ denotes the norm $\lVert N_{Y} \rVert_{L^1(M)}$ for $Y := \nabla(R - \rho)$. Thus, the conditions \eqref{eq:inner-prod-of-grad-rho-and-X-non-positive} and \eqref{eq:grad-R-minus-rho-in-L1-space} imply that $|V| \in L^1(M)$. Therefore, applying Proposition \ref{prop:Caminha-2011-triviality-of-div} we must have $\mathrm{div}(V) = 0$. In particular, from equation \eqref{eq:div-for-vector-field-V} we have
\begin{align*}
(R - \rho)^2 = \frac{1}{n} \langle \nabla \rho, X \rangle.
\end{align*}
But since $\langle \nabla \rho, X \rangle \leq 0$, we must have $R = \rho$, and hence $X$ is a Killing vector field.
\end{proof}

\vspace*{1mm}
Next, we complete the proof of the Corollary \ref{cor:complete-non-compact-case-under-finite-vol-and-Sobolev-condition} and Corollary \ref{cor:complete-non-compact-case-under-rho-and-R-minus-rho-non-negative}.

\begin{proof}[\textbf{Proof or Corollary \ref{cor:complete-non-compact-case-under-finite-vol-and-Sobolev-condition}}]
Since $(M, g, X, \rho)$ is a complete non-compact almost Yamabe soliton of dimension $n \geq 3$ with finite volume, using \eqref{eq:divergence-of-X-for-almost-Yamabe-soliton} and the condition \eqref{eq:div-of-X-is-in-Sobolev-space} we have that $(R - \rho) \in L^2(M)$ and $|\nabla(R - \rho)| \in L^1(M)$. Since $\rho \in L^2(M)$, we also have $R \in L^2(M)$ and so the condition \eqref{eq:grad-R-minus-rho-in-L1-space} is satisfied. Therefore, applying Theorem \ref{thm:main-1-complete-non-cpt-case} we conclude that $X$ is a Killing vector field.
\end{proof}

\vspace*{1mm}
\begin{proof}[\textbf{Proof of Corollary \ref{cor:complete-non-compact-case-under-rho-and-R-minus-rho-non-negative}}]
Note that the condition $R + |X| \in L^2(M)$ and the inequality $0 \leq \rho \leq R$ imply that $|X| \in L^2(M)$, $R \in L^2(M)$ and $\rho \in L^2(M)$. Since we also have $\langle \nabla \rho, X \rangle \leq 0$ and $|\nabla(R - \rho)| \in L^1(M)$, applying Theorem \ref{thm:main-1-complete-non-cpt-case} we conclude that $X$ is a Killing vector field.
\end{proof}

\vspace*{1.5mm}
\section*{Acknowledgements}
	
This work was supported in part by an Institute Post Doctoral Fellowship from the Indian Institute of Technology Bombay.

\vspace*{1.5mm}

\end{document}